\documentclass[11pt]{article}
\usepackage[T1]{fontenc}
\usepackage{lmodern}
\usepackage{amsmath,amsthm,amssymb,amsfonts}
\usepackage{enumerate}
\usepackage{epsfig}
\usepackage[dvipsnames,svgnames,table]{xcolor}
\usepackage{pgf,tikz}
\usetikzlibrary{calc}
\usepackage{fullpage}
\usepackage[colorlinks=true,linkcolor=RoyalBlue,urlcolor=RoyalBlue,citecolor=red]{hyperref}
\usepackage{booktabs}
\usepackage[nameinlink]{cleveref}
\usepackage{enumitem}
\usepackage{blkarray}
\usepackage{tikz-cd}

\newcommand{\comment}[1]{%
  \text{\phantom{(#1)}} \tag{#1}
}

\newtheorem{theorem}{Theorem}[section]
\newtheorem{proposition}[theorem]{Proposition}
\newtheorem{lemma}[theorem]{Lemma}

\newtheorem{conjecture}[theorem]{Conjecture}
\newtheorem{claim}[theorem]{Claim}

\theoremstyle{definition}
\newtheorem{definition}[theorem]{Definition}
\newtheorem{example}[theorem]{Example}

\theoremstyle{remark}
\newtheorem{remark}[theorem]{Remark}

\newcommand{\CC}{\mathbb{C}}

\newcommand{\QQ}{\mathbb{Q}}
\newcommand{\ZZ}{\mathbb{Z}}
\newcommand{\FF}{\mathbb{F}}

\DeclareMathOperator{\rank}{rank}

\DeclareMathOperator{\diag}{Diag}
\DeclareMathOperator{\Log}{Log}
\DeclareMathOperator{\Gr}{Gr}

\newcommand{\Asterisk}{\mathop{\scalebox{2}{\raisebox{-0.2ex}{$\ast$}}}}%

\tikzstyle{vertex}=[fill=black,circle,inner sep=0pt, minimum size=4pt]
\tikzstyle{edge}=[line width=1.5pt,black]

\title{Dilworth truncations and Hadamard products of linear spaces}

\author{
Dario Antolini\thanks{Dipartimento di Matematica, Università di Trento. Email: \texttt{dario.antolini-1@unitn.it}}
\and
Sean Dewar\thanks{School of Mathematics, University of Bristol. E-mail: \texttt{sean.dewar@bristol.ac.uk}}
\and
Shin-ichi Tanigawa\thanks{
Graduate School of Information Science and Technology, 
University of Tokyo,  
Email: \texttt{tanigawa@mist.i.u-tokyo.ac.jp}}}

\begin{document}
\date{}
\maketitle

\begin{abstract}
As a direct application of Dilworth truncations of polymatroids, we give short proofs of two theorems: Bernstein's characterisation of algebraic matroids coming from the Hadamard product of two linear spaces, and a formula for the dimension of the amoeba of a complex linear space by Draisma, Eggleston, Pendavingh, Rau, and Yuen.
We disprove Bernstein's conjecture on a characterisation of the algebraic matroids of Hadamard products of more than two linear spaces, by giving explicit counterexamples. 
\end{abstract}



\section{Introduction}

Let $\mathbb{F}$ be a field and consider the vector space $\mathbb{F}^m$ on which we fix a basis. The \emph{Hadamard product} of any $d$ vectors in $\mathbb{F}^m$ is the vector formed by their coordinate-wise multiplication: namely,
if $x_1,\ldots,x_d \in \mathbb{F}^m$ with $x_i=(x_i(j))_{j \in [m]}$ for each $i \in [d]$, their Hadamard product is given by
\begin{equation*}
    x_1 \ast x_2 \ast \cdots \ast x_d := \Big( x_1(j) \ x_2(j) \ \cdots \ x_d(j) \Big)_{j \in [m]}.
\end{equation*}
Given algebraic varieties $X_1,\ldots,X_d \subset \mathbb{F}^m$, their \emph{Hadamard product} is defined as
\begin{equation*}
    X_1 \ast \cdots \ast X_d := \overline{\big\{ x_1 \ast \cdots \ast x_d : x_1  \in X_1, \ldots , x_d \in X_d \big\}}
\end{equation*}
where the closure is taken with respect to the Zariski topology on $\mathbb{F}^m$.

Hadamard products of algebraic varieties were introduced by Cueto, Morton and Sturmfels \cite{CMS10:GeometryRBM} in the projective setting, while investigating Restricted Boltzmann Machines. These are algebraic statistical models corresponding to bipartite graphs which are building blocks for deep neural networks \cite{Mon16}. These models motivated the study of other algebraic statistical models related to Hadamard products of special projective varieties \cite{CTY10:Implicitization,MM17:DimensionKronecker}, as well as the study of algebro-geometric properties of Hadamard products of projective varieties. 
For a general overview on this latter direction, see the book
\cite{bocci2024hadamard}.

In this paper, we develop a matroidal approach to analyze the combinatorial aspect of Hadamard products of linear spaces via graph rigidity theory. This was initiated by Bernstein \cite{bernstein22}, who observed that, for any integer $n\geq 3$, the Cayley-Menger variety of $n$ points in dimension 2 is the Hadamard square of the linear space obtained as the image of the incidence matrix of an (arbitrary oriented) complete graph $K_n$. Since this matrix represents the graphic matroid of $K_n$, the 2-dimensional rigidity matroid on $K_n$ is understood as the algebraic matroid of an Hadamard square of a linear space. 

Motivated towards a new proof of Geiringer--Laman's theorem \cite{pollaczek-geiringer1927,Laman1970}, which is one of the representative results in graph rigidity theory, Bernstein gave the following characterisation of the algebraic matroid of the Hadamard product of two (complex) linear spaces.

\begin{theorem}[Bernstein \cite{bernstein22}]\label{thm:bernstein}
\label{thm:two_factors}
    Let $L_1, L_2 \subset \CC^m$ be linear spaces not contained in any coordinate hyperplane and let $r_1, r_2$ be the rank functions of the two corresponding linear matroids. Then, the algebraic matroid of the Hadamard product $L_1 * L_2$ is the matroid induced by the monotone submodular function $r_1 + r_2 - 1$.
\end{theorem}

\Cref{thm:bernstein} implies that the dimension of coordinate projections of Hadamard products of two linear spaces can be computed in deterministic polynomial time, assuming the rank oracles $r_1$ and $r_2$.
Bernstein's proof uses tropical geometry, which allows to transform the statement in a result about Minkowski sums of the Bergman fans of the two linear spaces. The same approach was adopted before in the computation of the number of realisations of a Laman graph \cite{capco}. The natural generalisation of this characterisation is the following conjecture by Bernstein.

\begin{conjecture}[Bernstein \cite{bernstein22}]\label{con:bernstein}
    Let $L_1, \ldots, L_d \subset \CC^m$ be linear spaces not contained in any coordinate hyperplane and let $r_1, \ldots, r_d$ be the rank functions of the corresponding linear matroids. Then, the algebraic matroid of the Hadamard product $L_1 * \cdots * L_d$ is the matroid induced by the monotone submodular function $r_1 + \cdots + r_d - (d-1)$.
\end{conjecture}

\subsection{Our contribution and structure of the paper}

We revisit Lov{\'a}sz and Yemini's approach \cite{lovasz82} to Dilworth truncations of polymatroids in order to give a new proof of Bernstein's theorem for the case of two linear spaces (\Cref{sec:hadamard}).
We also show that the same approach leads to a new proof of a result by Draisma, Eggleston, Pendavingh, Rau, and Yuen \cite{amoeba} on the dimension of the amoeba of a complex linear space. The key step is to reduce the computation of the (real) dimension of the amoeba to the computation of the (complex) dimension of an Hadamard product of two linear spaces (\Cref{sec:amoeba}).
On the other hand, the technique fails for the Hadamard product of three or more linear spaces. In fact, we provide two counterexamples to Bernstein's conjecture (\Cref{sec:counterexamples}).
However, we verify that the conjecture does hold for a generic choice of linear spaces (\Cref{sec:generic}).
While this result for generic linear spaces does follow from a special case of a result of Ballico \cite{BALLICO},
we present an elementary and simple alternative proof.
Moreover, for any choice of $n_1,\ldots,n_d$, we provide explicit examples of linear spaces of dimension $n_i$ such that their Hadamard product attains the maximum dimension $n_1 + \cdots + n_d + (d-1)$.

\section{Dilworth truncations}

\subsection{Combinatorial description}

It was Lov{\'a}sz and Yemini~\cite{lovasz82} who first pointed out the connection between the Geiringer–Laman theorem for rigidity and Dilworth truncations. Since then, this approach has been successfully applied to various rigidity problems, including the rigidity of symmetric or periodic frameworks~\cite{tangiawa}, point-line frameworks~\cite{jackson}, and scene analysis~\cite{whiteley}. In this section we review basic facts on polymatroids related to Dilworth truncations. 

\begin{definition}
    A \emph{polymatroid} is a pair $(E,r)$, where $E$ is a finite set and $r\colon 2^E \to \mathbb{Z}$ is a function with the following properties:
    \begin{itemize}
        \item $r$ is \emph{submodular}, i.e.\ $r(X) + r(Y) \geq r(X \cup Y) + r(X \cap Y)$ for all $X, Y \subset E$;
        \item $r$ is \emph{monotone}, i.e.\ $r(X) \leq r(Y)$ whenever $X \subset Y \subset E$;
        \item $r(\emptyset) = 0$.
    \end{itemize}
\end{definition}

\begin{example}
\label{example:matroid}
    If $(E,r)$ is a polymatroid satisfying $r(\{e\})\leq 1$ for all $e \in E$, then $(E,r)$ is a matroid defined by the rank function $r$.
\end{example}

\begin{example}
\label{example:linearpolymatroid}
    Let $E$ be a finite set, $\mathbb{F}$ a field and consider a collection ${\cal A} = \{ A_e \subset \mathbb{F}^m : e \in E\}$ of linear subspaces. For any $F \subset E$, define ${\cal A}_F = \langle A_e : e \in F\rangle$ and $r(F) = \dim {\cal A}_F$. Then, $(E,r)$ is a polymatroid.
\end{example}

\begin{definition}
    A \emph{linear polymatroid} over a field $\FF$ is any polymatroid arising as in \Cref{example:linearpolymatroid}.
\end{definition}

As mentioned in \Cref{example:matroid}, any matroid is a polymatroid.
Conversely, a polymatroid (or, equivalently, any integer-valued monotone submodular function) induces a matroid, see, e.g.~\cite[Proposition 11.1.7]{oxley}.

\begin{proposition}[\cite{edmonds1966submodular}]
\label{prop:induced}
For a polymatroid $(E,f)$, the collection
\[
{\cal I}_f:=\{F\subseteq E: |I|\leq f(I) \text{ for all $I\subseteq F$}\}
\]
forms a matroid on $E$, whose rank function is
\begin{equation}\label{eq:rank}
r_f(F)=\min\{|F\setminus I|+f(I): I\subseteq F\} \qquad (F\subseteq E).
\end{equation}
\end{proposition}
The matroid $(E,{\cal I}_f)$ given in \Cref{prop:induced} is called the \emph{matroid induced by $f$}, denoted by ${\cal M}_f$.

Recall that a polymatroid $(E,f)$ satisfies $f(\emptyset)=0$ by definition.
Dilworth truncations can be used to extend 
\Cref{prop:induced} to the case when $f(\emptyset)$ is arbitrary.
We first introduce Dilworth truncations in the combinatorial setting.
Our terminology follows Schrijver~\cite[Section 48]{schrijver2003combinatorial}.
\begin{definition}
The \emph{Dilworth truncation} of  a set function $f\colon2^E\rightarrow \mathbb{R}$ is the function $f^{\rm D}\colon2^E\rightarrow \mathbb{R}$ defined as
\begin{equation}\label{eq:dilworth}
f^{\rm D}(F)=
\begin{cases} 
\min \left\{ \sum_{i=1}^t f(F_i) : \{F_1,\ldots,F_t\} \text{ is a partition of } F \right\} & (\text{if } F \neq \emptyset),\\ 
0 & (\text{if } F=\emptyset),
\end{cases}
\end{equation}
where (and hereafter) a partition of $F$ means a collection of \emph{nonempty} disjoint subsets of $F$
whose union is $F$.
\end{definition}

The following theorem can be pieced together from results of Edmonds~\cite{edmonds1970submodular} which were later discovered independently by Dunstan~\cite{dunstan1976matroids}.
Since the statement given below is an adapted version for our purposes, we briefly outline how it can be verified for the sake of completeness.

\begin{theorem}
\label{thm:combinatorial_dilworth}
    Let $E$ be a finite set and let $f \colon 2^E \rightarrow \mathbb{Z}$ be a set function.
    If $f$ is monotone, submodular and non-negative on $E \setminus \{\emptyset\}$, then $f^{\rm D}$ is monotone submodular, and hence $(E,f^{\rm D})$ is a polymatroid.
    Moreover, the matroid ${\cal M}_{f^{\rm D}}=(E,{\cal I}_{f^{\rm D}})$ induced by $f^{\rm D}$ satisfies
    \begin{equation}\label{eq:independent_sets_dilworth}
        {\cal I}_{f^{\rm D}}=\{F\subseteq E: |I|\leq f(I) \text{ for all $I\subseteq F$ with $I\neq \emptyset$}\},
    \end{equation}
    and its rank function is given by
    the combination of (\ref{eq:rank}) and (\ref{eq:dilworth}) for $f$, i.e.,
    \begin{align} \nonumber
        r_{f^{\rm D}}(F)=f^{\rm MD}(F):&=\min\left\{|F_0|+f^{\rm D}(F\setminus F_0): F_0\subseteq F \right\} \\
        &=\min\left\{|F_0|+\sum_{i=1}^t f(F_i): F_0\subseteq F, \{F_1, \dots, F_t\} \text{ is a partition of } F\setminus F_0\right\} \label{eq:dilworth_matroid}
    \end{align}
    for any nonempty $F\subseteq E$. 
\end{theorem}

\begin{proof}
    That $f^D$ is submodular is proved in \cite[Theorem 48.2]{schrijver2003combinatorial}.
    It is immediate from $f$ being monotone that $f^D$ must also be monotone.
    Since $f^D(\emptyset) = 0$ by definition, $(E,f^{\rm D})$ is a polymatroid.

    We now prove equality (\ref{eq:independent_sets_dilworth}).
    Since $(E,f^D)$ is a polymatroid, we have matroid ${\cal M}_{f^D}$ whose independent set family ${\cal I}_{f^D}$ is given by 
    \[
    {\cal I}_{f^D}=\{F\subseteq E: |I|\leq f^D(I) \text{ for all $I\subseteq F$}\}.
    \]
    Let 
    \[
    {\cal I}'=\{F\subseteq E: |I|\leq f(I) \text{ for all $I\subseteq F$ with $I\neq \emptyset$}\}.
    \]
    Since $f^D(I)\leq f(I)$ for any nonempty $I$, we have
    ${\cal I}_{f^D}\subseteq {\cal I}'$.
    To see the reverse inclusion,
    pick any nonempty $F\in {\cal I}'$.
    Then, any subset of $F$ is included in ${\cal I}'$ by the definition of ${\cal I}'$.
    Consider a nonempty set $I\subseteq F$ and let $\{I_1,\dots, I_k\}$ be a partition of $I$ that attains $f^D(I)$.
    Then, $|I_i|\leq f(I_i)$ holds by $I_i\in {\cal I}'$, and we get $|I|=\sum_i |I_i|\leq \sum_i f(I_i)=f^D(I)$,
    implying $F\in {\cal I}_{f^D}$.

    The rank function for $\mathcal{M}_{f^D}$ follows directly from \Cref{prop:induced} and (\ref{eq:dilworth}) for $f$.
\end{proof}

Observe that if a submodular function $f$ satisfies $f(\emptyset)\geq 0$,
the submodularity implies $f^{\rm D}(F)=f(F)$ for all $F\subseteq E$ with $F\neq \emptyset$, and hence ${\cal M}_{f^{\rm D}}={\cal M}_f$.
Hence, there is no ambiguity even if we call ${\cal M}_{f^{\rm D}}$ 
the \emph{matroid induced by $f$} and denote it simply by ${\cal M}_f$
for a monotone submodular function $f$ that is non-negative on $E \setminus \{\emptyset\}$ (which may not satisfy $f(\emptyset)\geq 0$).

\begin{example}\label{ex:canonical}
Suppose $f\colon2^E\rightarrow \mathbb{Z}$ is a monotone submodular function
with $f(\emptyset)=0$
and $k$ is an integer.
The function $g:=f-k$ is monotone submodular.
As above, we denote the matroid ${\cal M}_{g^D}$ simply by ${\cal M}_g$ with a slight abuse of notation. This has independent set family:
\[
{\cal I}_{g^{\rm D}}=\{F\subseteq E: |I|\leq f(I)-k \text{ for all $I\subseteq F$ with $I\neq \emptyset$}\}.
\]
\end{example}

\subsection{Geometric description}
\label{subsection:geometric_dilworth}

A key fact due to Lov{\'a}sz is that, if the initial polymatroid ${\cal M} =(E,f)$ is linear,
then there is a canonical way to get a linear representation of ${\cal M}_{g}$. We first recall a standard description of genericity.

\begin{definition}
\label{def:generic}
    Let $V$ be an algebraic variety over $\mathbb{R}$ or $\mathbb{C}$, and let $P$ be a property for points in $V$.
    We say that $P$ \emph{holds for a generic point} if there exists a non-empty Zariski open subset $U \subset V$ such that property $P$ holds for every point in $U$.
    As an abuse of terminology, we say that any point in $V$ for which $P$ holds is itself a \emph{generic point}.
\end{definition}

The Grassmannian $\Gr_k(\mathbb{F}^n)$ is the variety of $k$-dimensional linear subspaces of $\mathbb{F}^n$.
A family of linear spaces $L_1, \ldots, L_d \subset \mathbb{F}^n$ is said to be \emph{generic} if the point $(L_1,\ldots,L_d)$ is a generic point of the variety $\Gr_{n_1}(\mathbb{F}^n) \times \cdots \times \Gr_{n_d}(\mathbb{F}^n)$.

A geometric description of Dilworth truncation is given in the following theorem.

\begin{theorem}[Lov\'asz \cite{lovasz1977flats}]\label{thm:dil1}
    Let $\mathbb{F} \in \{\mathbb{R},\mathbb{C}\}$ and let $(E,f)$ be a linear polymatroid with a linear representation $\{A_e \subset \mathbb{F}^n: e\in E\}$ for some positive integer $n$.
    Suppose that $\dim A_e\geq k$.
    Then, for any generic codimension $k$ linear subspace $H \subset \FF^n$ and any $F\subseteq E$,
    we have
    \begin{equation*}
        \dim \langle A_e \cap H: e \in F \rangle = \min \left\{ \sum_{i=1}^t (\dim \langle A_e: e\in F_i\rangle - k ): \{F_1,\ldots,F_t\} \text{ is a partition of } F \right\}.
    \end{equation*}
    Thus, $\{A_e\cap H \subset \mathbb{F}^n: e\in E\}$ is a linear representation of $(E,g^{\rm D})$ for $g=f-k$.
\end{theorem}
A proof of \Cref{thm:dil1} is given in \cite{Raz2019}. A relatively shorter proof is given in \Cref{appendix} for readers convenience.

\section{Hadamard Products of Linear Spaces}\label{sec:hadamard}

\subsection{Algebraic matroids of Hadamard products of linear spaces}
We first collect basic facts on algebraic matroids and Hadamard products of linear spaces. 

\begin{definition}
Let $V\subseteq \mathbb{C}^m$ be an irreducible algebraic variety,
and let $\pi_E\colon \CC^m \to \CC^E$ be the coordinate projection on the variables indexed by $E\subseteq [m]$.
The algebraic matroid ${\cal M}(V)$ of $V$ is the matroid $([m],{\cal I}_V)$ with 
\[
\mathcal{I}_V = \{
E \subseteq [m] : \overline{\pi_E(V)} = \mathbb{C}^E
\}.
\]
If $V$ is the Zariski closure of the image of a polynomial map $f\colon\mathbb{C}^n \rightarrow \mathbb{C}^m$,
then for a generic point $p\in \mathbb{C}^{n}$, 
the matroid $\mathcal{M}(V)$ admits a linear representation by 
the Jacobian matrix $J f(p)$ at $p$; see, for example, \cite{rosen2020algebraic,rosen2025linearizing}.
\end{definition}

Let $L_1,\dots, L_{d}$ be linear spaces in $\mathbb{C}^m$
with $n_i = \dim L_i$.
Since the variety $L_1\ast \dots \ast L_d$  is the Zariski closure of the image of a polynomial map $h\colon \CC^{n_1} \times \cdots \times \CC^{n_d} \to \CC^m$,   $L_1\ast \dots \ast L_d$ is irreducible and its algebraic matroid is linearly represented by the Jacobian of $h$.
The following fact gives an explicit linear representation.
\begin{proposition}\label{prop:tangent}
Let $\mathbb{F} \in \{\mathbb{R}, \mathbb{C}\}$,
let $L_1,\dots, L_d\subset \mathbb{F}^m$ be linear spaces, and let $x_i$ be a generic point in $L_i$ for each $i \in [d]$.
Then the tangent space of $L_1\ast \dots \ast L_d$ at the point $x_1\ast \dots \ast x_d$ is 
\[
\left(\Asterisk_{i\neq 1} x_i \right)\ast L_1+ \left(\Asterisk_{i\neq 2} x_i \right)\ast L_2+ \dots +\left(\Asterisk_{i\neq d} x_i \right)\ast L_d.
\]
Hence,
if $Y_i \in \mathbb{F}^{m \times \dim L_i}$ is a matrix whose column space is $L_i$ for each $i \in [d]$,
then the tangent space of $L_1\ast \dots \ast L_d$ at $x_1\ast \dots \ast x_d$ is the column space of the matrix
\begin{equation*}
    \left( \left(\prod_{i \neq 1}{\rm Diag}(x_i) \right) Y_1 \ \Bigg\vert \ \left(\prod_{i \neq 2}{\rm Diag}(x_i) \right) Y_2 \ \Bigg\vert \ \cdots \ \Bigg\vert \ \left(\prod_{i \neq d}{\rm Diag}(x_i) \right) Y_d\right)
\end{equation*}
where $\diag(x)$ denotes the diagonal matrix with diagonal entries $x\in \mathbb{F}^m$.
\end{proposition}
\begin{proof}
This is a direct consequence of the more general version of Terracini's Lemma for Hadamard products,
see, e.g.,~\cite{bocci2024hadamard,BCK17}.
In the case of linear spaces, one can also verify the statement by an elementary argument. 
Indeed, if $d=1$, then the statement is trivial.
The case when $d\geq 2$ follows by the chain rule and induction. 
\end{proof}

For a family of matroids on the same ground set $E$,
the \emph{weak order} is defined as follows:
for ${\cal N}_1$ and ${\cal N}_2$ in the family,
${\cal N}_1 \preceq {\cal N}_2$ if and only if 
every independent set in ${\cal N}_1$ is independent in ${\cal N}_2$. Equivalently, $r_{{\cal N}_1}(F)\leq r_{{\cal N}_2}(F)$ for all $F\subseteq E$. In particular, if $X \subset \FF^m$ is an algebraic variety and $\mathcal{N}$ is a matroid with ground set $[m]$, then
\[
    {\cal M}(X) \preceq {\cal N} \iff \dim \pi_F(X) \leq r_{{\cal N}}(F) \qquad \text{for all } F \subset [m].
\]

The following lemma showcases an important link between Hadamard products of linear spaces and Dilworth truncations.

\begin{lemma}\label{lem:improvedmain}
    Let $L_1,\ldots,L_d \subset \mathbb{C}^m$ be linear spaces not contained in any coordinate hyperplane.
    For any $k$ with $1<k<d$,
    fix $r_a$ to be the rank function of $\mathcal{M}\left(L_1\ast L_2\ast \dots \ast L_k\right)$ and
    $r_b$ to be the rank function of $\mathcal{M}\left( L_{k+1}\ast L_{k+2} \ast \dots \ast L_d\right)$.
    Then
    \begin{equation*}
        \mathcal{M}(L_1 \ast \cdots \ast L_d) \preceq \mathcal{M}_{r_a + r_b -1}.
    \end{equation*}
\end{lemma}
\begin{proof}
We first show that the rank of $\mathcal{M}(L_1 \ast \cdots \ast L_d)$ is bounded by $r_a([m])+r_b([m])-1$.
To see this, pick any generic $x_i\in L_i$ and observe  that 
\begin{equation}\label{eq:common}
\Asterisk_{j\in I} x_{j}\in \bigcap_{s\in I} \left(\Asterisk_{j\in I\setminus \{s\}} x_j\right)\ast L_{s}
\end{equation}
for any $I\subseteq [m]$.

\Cref{prop:tangent} implies that 
the tangent space of $L_1\ast \dots \ast L_d$ at the point 
$x_1\ast \dots \ast x_d$ is 
\begin{equation}\label{eq:tangent_decomp}
\left(\Asterisk_{i=k+1}^{d} x_i \right) \ast TV_a+\left(\Asterisk_{i=1}^k x_i \right) \ast TV_b
\end{equation}
where $TV_a$ denotes the tangent space of $L_1\ast \dots \ast L_k$
at $x_1\ast \dots \ast x_k$
and $TV_b$ denotes the tangent space of $L_{k+1}\ast \dots \ast L_d$
at $x_{k+1}\ast \dots \ast x_d$.
By (\ref{eq:common}), 
$TV_a$ and $TV_b$ contain $\Asterisk_{i=1}^k x_i$ and $\Asterisk_{i=k+1}^{d} x_i$, respectively, and hence 
both $(\Asterisk_{i=k+1}^{d} x_i) \ast TV_a$ and $(\Asterisk_{i=1}^k x_i) \ast TV_b$ contains $\Asterisk_{i=1}^{\rm d} x_i$.
Since $L_i$ is not contained in a coordinate hyperplane, $\Asterisk_{i=1}^{d} x_i$ is non-zero.
Thus, the dimension of (\ref{eq:tangent_decomp}) is bounded by $r_a([m])+r_b([m])-1$,
and the rank of $\mathcal{M}(L_1 \ast \cdots \ast L_d)$ is indeed bounded by $r_a([m])+r_b([m])-1$.

To see the rank bound for any subset $E\subseteq [m]$,
recall that $\pi_E$ is the coordinate projection to $\mathbb{C}^E$.
Hence, $\pi_E$ and $\ast$ commute,
and $\overline{\pi_E(L_1\ast \dots \ast L_d)}=\overline{\pi_E(L_1)\ast \dots \ast \pi_E(L_d)}$ holds.
Therefore, we can apply the same argument
to the linear spaces $\pi_E(L_1), \dots, \pi_E(L_d) \subset \CC^E$.
\end{proof}

\subsection{Proof of \texorpdfstring{\Cref{thm:bernstein}}{Bernstein's theorem}}

We now provide a new proof of Berstein's theorem (\Cref{thm:bernstein}) using Dilworth truncations.

\begin{proof}[Proof of \Cref{thm:bernstein}]
For $i=1,2$, let $n_i=\dim L_i$ and $Y_i$ be a $m\times n_i$ matrix representing $L_i$, i.e., $L_i$ is the column space of $Y_i$.
For $e\in [m]$, we denote by $y_{i,e}$ the $e$-th row of $Y_i$.
Note that $Y_i$ is a matrix representation of the algebraic matroid ${\cal M}(L_i)$, where $e\in [m]$ is associated with $y_{i,e}$. Since $L_i$ is not contained in any coordinate hyperplane, we have that $y_{i,e} \neq {\bf 0}$ for all $e \in E$.

Pick generic elements $p_1\in \mathbb{C}^{n_1}$ and $p_2\in \mathbb{C}^{n_2}$.
Then, $x_i:=Y_ip_i$ is a generic element in $L_i$.
\Cref{prop:tangent} implies that 
the tangent space of $L_1\ast L_2$ at $x_1\ast x_2$ is 
$x_2\ast L_1+x_1\ast L_2$, or equivalently, in a matrix form, 
\begin{equation}\label{eq:berstein1}
\Big( ~ {\rm Diag}(x_2)Y_1 \ \Big\vert \ {\rm Diag}(x_1)Y_2 ~ \Big).
\end{equation}
To apply \Cref{thm:dil1},
we consider $A_e, H \subset \mathbb{C}^{n_1}\times \mathbb{C}^{n_2}$ as follows.
For $e\in [m]$, let 
\begin{equation}\label{eq:berstein2}
A_e=\langle y_{1,e}\rangle \oplus \langle y_{2,e}\rangle\subseteq \mathbb{C}^{n_1}\times \mathbb{C}^{n_2}
\end{equation}
and set
    \begin{equation*}
        H := \left\{ q=(q_1,q_2) \in \mathbb{C}^{n_1}\times \mathbb{C}^{n_2} : \sum_{j=1}^{n_1} p_1(j) q_1(j) - \sum_{j=1}^{n_2} p_2(j) q_2(j)
        =0    
        \right\}.
    \end{equation*}
Since $p_1$ and $p_2$ are generic vectors, $H$ is a generic hyperplane;
indeed, the orthogonal vector to $H$ is $(p_1,-p_2)$ which is generic because $p_1$ and $p_2$ are generic.
Observe that
\begin{equation}\label{eq:berstein3}
A_e\cap H=\langle x_2(e) y_{1,e} + x_1(e) y_{2,e}\rangle.
\end{equation}
Indeed, by (\ref{eq:berstein2}), an arbitrary element in $A_e$ is denoted by 
$\lambda_1y_{1,e}+\lambda_2y_{2,e}$ for $\lambda_1,\lambda_2\in \mathbb{C}$,
and  $\lambda_1y_{1,e}+\lambda_2y_{2,e}$ belongs to $H$ if and only if 
$\lambda_1 \left( \sum_{j=1}^{n_1} p_1(j) y_{1,e}(j) \right)=\lambda_2 \left( \sum_{j=1}^{n_2} p_2(j) y_{2,e}(j) \right)$, or equivalently
$\lambda_1 x_1(e) = \lambda_2 x_2(e)$
by $x_i(e)=(Y_ip_i)(e)=\sum_{j=1}^{n_i}  y_{i,e}(j) p_i(j)$ for $i=1,2$.
Thus, (\ref{eq:berstein3}) holds.

Comparing (\ref{eq:berstein1}) and (\ref{eq:berstein3}), 
we see that $\{A_e\cap H :e \in [m]\}$ is a linear representation of ${\cal M}(L_1\ast L_2)$.
Since $\dim A_e \geq 1$ for all $e \in E$, we can apply \Cref{thm:dil1}. Hence, for any nonempty $F\subseteq [m]$,
\begin{align*}
&r_{{\cal M}(L_1\ast L_2)}(F)\\
&=\dim\langle A_e\cap H: e\in F\rangle \\
&=\min\left\{\sum_{i=1}^t(\dim\langle A_e: e\in F_i\rangle-1): 
\{F_1,\dots, F_t\} \text{ is a partition of $F$}\right\} \\
&=\min\left\{\sum_{i=1}^t(\dim\langle y_{e,1}: e\in F_i\rangle+\dim\langle y_{e,2}: e\in F_i\rangle-1): 
\{F_1,\dots, F_t\} \text{ is a partition of $F$}\right\} \\
&=\min\left\{\sum_{i=1}^t(r_1(F_i)+r_2(F_i)-1): 
\{F_1,\dots, F_t\} \text{ is a partition of $F$}\right\} \\
&=(r_1+r_2-1)^{\rm D}(F)
\end{align*}
Hence, ${\cal M}(L_1\ast L_2)$ is equal to ${\cal M}_{r_1+r_2-1}$.
\end{proof}

\begin{remark}
    If, for instance, $L_1$ is contained in some coordinate hyperplane $\{x_i = 0\}$, then $i$ is a loop of both ${\cal M}(L_1)$ and ${\cal M}(L_1 \ast L_2)$. In this case, the matroid $\mathcal{M}(L_1  \ast L_2)$ is the unique matroid formed from $\mathcal{M}(\pi_{[m]\setminus \{i\}}(L_1) \ast \pi_{[m]\setminus \{i\}}(L_2))$ by adding the loop $i$. 
    Using this observation, we can now deal with the general case.
    If $S_1,S_2 \subset [m]$ are the sets of loops of $\mathcal{M}(L_1)$ and $\mathcal{M}(L_2)$ respectively,
    then $\mathcal{M}(L_1 \ast L_2)$ is the unique matroid formed from $\mathcal{M}(L_1' \ast L_2')$ by adding the elements $S_1 \cup S_2$ as loops,
    where $L_1' := \pi_{[m] \setminus (S_1 \cup S_2)}(L_1)$ and $L_2' := \pi_{[m] \setminus (S_1 \cup S_2)}(L_2)$.
\end{remark}

Assuming an oracle for computing each value of $f\colon2^E\rightarrow \mathbb{Z}$, 
there is a deterministic polynomial time algorithm
for computing each value of $f^{\rm D}$ or $f^{\rm MD}$, see, e.g.~\cite[Chapter 48]{schrijver2003combinatorial}.
Hence, 
\Cref{thm:bernstein} implies that the dimension of $L_1\ast L_2$ can be computed in polynomial time deterministically.

\section{Counterexamples to \texorpdfstring{\Cref{con:bernstein}}{main conjecture}}\label{sec:counterexamples}

We give two types of counterexamples to \Cref{con:bernstein}.

\subsection{First counterexample}
The first counterexample is obtained by a combinatorial consideration of Dilworth truncations. We start with the following observation.

\begin{proposition}\label{prop:dilworth_decomp}
Let $f_1, f_2,\dots, f_k$ be integer-valued monotone submodular functions on a finite set $E$ and $\ell$ be a non-negative integer.
Then, 
\begin{equation}\label{eq:dilworth_decomp}
{\cal M}_{\sum_{i=1}^k f_i^{\rm MD}-\ell}\preceq {\cal M}_{\sum_{i=1}^k f_i-\ell}.
\end{equation}
Moreover, the equality holds in (\ref{eq:dilworth_decomp}), if $f_1,\dots, f_k$ have the property that,
for each nonempty $F\subseteq E$,
there is a common minimizer $\{F_0, F_1,\dots, F_t\}$ in (\ref{eq:dilworth_decomp})
that attains 
$f_i^{\rm MD}(F)$  for all $i=1,\dots, k$ simultaneously.
\end{proposition}
\begin{proof}
Since $f_i^{\rm D}(F)\leq f_i(F)$ for any nonempty $F\subseteq E$,
(\ref{eq:dilworth_decomp}) follows from \Cref{thm:combinatorial_dilworth}.

To see the latter claim, suppose $f_1,\dots, f_k$ have the property described in the statement, and consider any independent set $F$ of ${\cal M}_{\sum_{i=1}^k f_i-\ell}$. 
Then, for any nonempty $I\subseteq F$,
there is a partition $\{I_1,\dots, I_t\}$ of $I$ that satisfies $f_i^{\rm MD}(I)=|I_0|+\sum_{j=1}^t f_i(I_j)$
 for all $i=1,\dots, k$ simultaneously.
Since $I_j$ is independent in ${\cal M}_{\sum_{i=1}^k f_i-\ell}$,
we have $|I_j|\leq \sum_{i=1}^k f_i(I_j)-\ell$ for $j=1,\dots, t$.
Thus, 
$|I|=\sum_{j=0}^t |I_j|\leq |I_0|+\sum_{j=1}^t (\sum_{i=1}^k f_i(I_j)-\ell)
\leq \sum_{i=1}^kf_i^{\rm MD}(I) -t\ell\leq \sum_{i=1}^kf_i^{\rm MD}(I) -\ell$,
and $F$ is indeed independent in  ${\cal M}_{\sum_{i=1}^k f_i^{\rm MD}-\ell}$.
\end{proof}
In general, the inequality in (\ref{eq:dilworth_decomp}) can be strict,
and this strict inequality will lead to a counterexample to \Cref{con:bernstein}.
Indeed, let $L_1,L_2, L_3$ be linear spaces and $r_1,r_2, r_3$ be the rank functions of the corresponding algebraic matroids.
Then, \Cref{con:bernstein} implies that ${\cal M}(L_1\ast L_2\ast L_3)$ is equal to ${\cal M}_{r_1+r_2+r_3-2}$.
On the other hand, \Cref{lem:improvedmain} implies that 
${\cal M}(L_1\ast L_2\ast L_3)$ is at most  
the matroid induced by $r_{{\cal M}(L_1\ast L_2)}+r_{{\cal M}(L_3)}-1$, or equivalently, 
\[
{\cal M}(L_1\ast L_2\ast L_3) \preceq {\cal M}_{(r_1+r_2-1)^{\rm MD}+r_3-1}
\]
by \Cref{thm:bernstein}.
In general, ${\cal M}_{(r_1+r_2-1)^{\rm MD}+r_3-1}$ can be strictly smaller than 
${\cal M}_{r_1+r_2+r_3-2}$ in the weak order, giving counterexamples to \Cref{con:bernstein}.

An instance distinguishing ${\cal M}_{(r_1+r_2-1)^{\rm MD}+r_3-1}$ and ${\cal M}_{r_1+r_2+r_3-2}$ appears naturally in graph rigidity problems
under cylindrical normed spaces, see \cite[Theorem 52]{KITSON2020}, or \cite[Theorem 3.10]{dewar2023rigidgraphscylindricalnormed}.
Here we give an example. 

Let $k$ and $\ell$ be integers with $k,\ell\geq 0$ and $2k-\ell > 0$.
For a graph $G=(V,E)$, define a set function $c_{k,\ell}\colon 2^E \to \ZZ$ by
\[
c_{k,\ell}(F)=k|V(F)|-\ell\qquad (F\subseteq E),
\]
where $V(F)$ denotes the set of vertices adjacent to edges in $F$.
Then, $c_{k,\ell}$ is a monotone submodular function
and the matroid ${\cal M}_{c_{k,\ell}}$ induced by $c_{k,\ell}$ is known as the \emph{$(k,\ell)$-count matroid} (or \emph{$(k,\ell)$-sparsity matroid}) of $G$.
Well known examples are 
the cases when $(k,\ell)=(1,1), (1,0), (2,3)$,
where the corresponding count matroids are 
the graphic matroid, the bicircular matroid, and the generic 2-dimensional rigidity matroid of $G$, respectively.
The rank function of ${\cal M}_{c_{k,\ell}}$ is given by $c_{k,\ell}^{\rm MD}$, cf.~Equation (\ref{eq:dilworth_matroid}).

\begin{proposition}\label{prop:count_matroid}
Let $a, b, \ell$ be non-negative integers such that $a+2b > \ell$.
For a graph $G=(V,E)$, 
the matroid induced by $ac_{1,1}^{\rm MD}+bc_{1,0}^{\rm MD}-\ell$ is the $(a+b,a+\ell)$-count matroid. 
\end{proposition}
\begin{proof}
$c_{1,1}^{\rm MD}$ and $c_{1,0}^{\rm MD}$ are the rank functions of 
the graphic and the bicircular matroids, respectively.
For any $F\subseteq E$, let 
$F_0$ be the union of the edge sets of all cycle-free connected components of the graph $(V,F)$ and let 
$\{F_1,\dots, F_t\}$ be the partition of $F\setminus F_0$ such that each $F_i$ is the edge set of a connected component of $(V,F)$ having some cycle.
Then, $c_{1,\ell}^{\rm MD}(F)=|F_0|+\sum_{i=1}^t c_{1,\ell}(F_i)$,
and $\{F_0,F_1,\dots, F_t\}$ is a common minimizer of $c_{1,\ell}^{\rm MD}(F)$ for all $\ell=0,1$.
Hence, by \Cref{prop:dilworth_decomp}, 
the matroid induced by $ac_{1,1}^{\rm MD}+bc_{1,0}^{\rm MD}-\ell$ is equal to that
induced by $ac_{1,1}+bc_{1,0}-\ell$.
By the decomposition $c_{a+b,a+\ell}=ac_{1,1}+bc_{1,0}-\ell$, the statement follows.
\end{proof}

\begin{remark}
A much stronger statement than \Cref{prop:count_matroid} is stated in~\cite[Proposition A.2.2.]{whiteley1996some} without a proof,
but this stronger statement is not correct as first pointed out by T.~Jord{\'a}n.
The present discussion on count matroids is based on~\cite[Section 2]{katoh2009infinitesimal} due to Katoh and the third author.
\end{remark}

We now construct a counterexample based on the graphic matroid (the $(1,1)$-sparsity matroid) and the bicircular matroid (the $(1,0)$-sparsity matroid).
It is a well-known fact that those two matroids are linearly representable over $\mathbb{C}$; see, e.g.,~\cite[Section 5]{oxley} for graphic matroids and \cite[Section 6.10]{oxley} or \cite[Section 2]{COULLARD1991223} for bicircular matroids. 
Moreover, the matrices representing these matroids have all non-zero rows, and hence their column space is not contained in any coordinate hyperplane as long as the underlying graph has no loop.
For a graph $G=(V,E)$ and $\ell=0,1$, let $X_{1,\ell}$ be a matrix representation of the $(1,\ell)$-sparsity matroid of $G$ as a row matroid,
and let $L_{1,\ell}$ be the column space of $X_{1,\ell}$.

By \Cref{thm:bernstein} and \Cref{prop:count_matroid},
we have
\begin{equation}
\label{eqn:laman}
    {\cal M}(L_{1,1}\ast L_{1,1})={\cal M}_{c_{1,1}^{\rm MD}+c_{1,1}^{\rm MD}-1}={\cal M}_{c_{2,3}},
\end{equation}
 meaning that the rank function of ${\cal M}(L_{1,1}\ast L_{1,1})$ is 
$c_{2,3}^{\rm MD}$.
Note that this is exactly  Geiringer-Laman's theorem.
Hence, we obtain
\begin{align*}
{\cal M}(L_{1,1}\ast L_{1,1}\ast L_{1,0})
&\preceq {\cal M}_{r_{ {\cal M}(L_{1,1}\ast L_{1,1})}+r_{{\cal M}(L_{1,0})}-1}
\comment{by \Cref{lem:improvedmain}} \\
&={\cal M}_{c_{2,3}^{\rm MD}+c_{1,0}^{\rm MD}-1} \comment{by (\ref{eqn:laman})}\\
&\preceq {\cal M}_{c_{2,3}+c_{1,0}-1} \comment{by \Cref{prop:dilworth_decomp}} \\
&={\cal M}_{c_{3,4}} \\
&={\cal M}_{c_{1,1}^{\rm MD}+c_{1,1}^{\rm MD}+c_{1,0}^{\rm MD}-2} \comment{by \Cref{prop:count_matroid}}\\
&={\cal M}_{r_{\cal M}(L_{1,1})+r_{\cal M}(L_{1,1})+r_{\cal M}(L_{1,0})-2}.
\end{align*}
Bernstein's conjecture (\Cref{con:bernstein}) claims the equality throughout the above relations.
However, this is not the case in general. 
Consider for example the graph $G=(V,E)$ given in \Cref{fig:34eg}. 
The edge set $E$ is a basis of $\mathcal{M}_{c_{3,4}}$, but $E$ does not contain any subset which is both an independent set of the $(2,3)$-count matroid and which spans all the vertices of $G$.
This in turn implies that, for this graph,
the rank of ${\cal M}_{c_{3,4}}$ is $3|V|-4$ whereas 
the rank of ${\cal M}_{c_{2,3}^{\rm MD}+c_{1,0}^{\rm MD}-1}$ is at most $2|V|-4+|V|-1=3|V|-5$.
Thus, for this graph $G$, 
${\cal M}(L_{1,1}\ast L_{1,1}\ast L_{1,0})$ is strictly smaller than 
the matroid induced by 
$r_{\cal M}(L_{1,1})+r_{\cal M}(L_{1,1})+r_{\cal M}(L_{1,0})-2$ in the weak order.

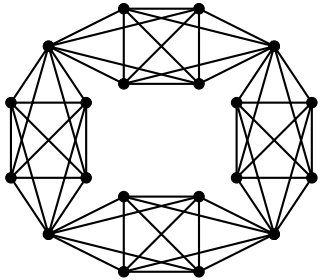
\begin{figure}[t]
 \centering
 \begin{tabular}{  l  r  }

 \begin{tikzpicture}[scale=0.5]
  \clip (-6.25,-6.5) rectangle (8.25cm,1.5cm); 
  \coordinate (A1) at (-1,-1);
  \coordinate (A2) at (1,-1);
  \coordinate (A3) at (1,1);
  \coordinate (A4) at (-1,1);
  \coordinate (A5) at (-3,0);
  \coordinate (A6) at (3,0);
	
  \draw[thick] (A1) -- (A5) -- (A4) -- (A2) -- (A3) -- (A1) -- (A4) -- (A3) -- (A6) --
	(A2) -- (A1);
	\draw[thick] (A2) -- (A5) -- (A3);
	\draw[thick] (A4) -- (A6) -- (A1);
  
  \node[draw,circle,inner sep=1.4pt,fill] at (A1) {};
  \node[draw,circle,inner sep=1.4pt,fill] at (A2) {};
  \node[draw,circle,inner sep=1.4pt,fill] at (A3) {};
  \node[draw,circle,inner sep=1.4pt,fill] at (A4) {};
  \node[draw,circle,inner sep=1.4pt,fill] at (A5) {};
  \node[draw,circle,inner sep=1.4pt,fill] at (A6) {};
	
	\coordinate (B1) at (-4,-3.5);
  \coordinate (B2) at (-2,-3.5);
  \coordinate (B3) at (-2,-1.5);
  \coordinate (B4) at (-4,-1.5);
  \coordinate (B5) at (-3,0);
  \coordinate (B6) at (-3,-5);
	
  \draw[thick] (B1) -- (B5) -- (B4) -- (B2) -- (B3) -- (B1) -- (B4) -- (B3) -- (B6) --
	(B2) -- (B1);
	\draw[thick] (B2) -- (B5) -- (B3);
	\draw[thick] (B4) -- (B6) -- (B1);
  
  \node[draw,circle,inner sep=1.4pt,fill] at (B1) {};
  \node[draw,circle,inner sep=1.4pt,fill] at (B2) {};
  \node[draw,circle,inner sep=1.4pt,fill] at (B3) {};
  \node[draw,circle,inner sep=1.4pt,fill] at (B4) {};
  \node[draw,circle,inner sep=1.4pt,fill] at (B5) {};
  \node[draw,circle,inner sep=1.4pt,fill] at (B6) {};
	
	\coordinate (C1) at (-1,-6);
  \coordinate (C2) at (1,-6);
  \coordinate (C3) at (1,-4);
  \coordinate (C4) at (-1,-4);
  \coordinate (C5) at (-3,-5);
  \coordinate (C6) at (3,-5);
	
  \draw[thick] (C1) -- (C5) -- (C4) -- (C2) -- (C3) -- (C1) -- (C4) -- (C3) -- (C6) --
	(C2) -- (C1);
	\draw[thick] (C2) -- (C5) -- (C3);
	\draw[thick] (C4) -- (C6) -- (C1);
  
  \node[draw,circle,inner sep=1.4pt,fill] at (C1) {};
  \node[draw,circle,inner sep=1.4pt,fill] at (C2) {};
  \node[draw,circle,inner sep=1.4pt,fill] at (C3) {};
  \node[draw,circle,inner sep=1.4pt,fill] at (C4) {};
  \node[draw,circle,inner sep=1.4pt,fill] at (C5) {};
  \node[draw,circle,inner sep=1.4pt,fill] at (C6) {};
	
	\coordinate (D1) at (4,-3.5);
  \coordinate (D2) at (2,-3.5);
  \coordinate (D3) at (2,-1.5);
  \coordinate (D4) at (4,-1.5);
  \coordinate (D5) at (3,0);
  \coordinate (D6) at (3,-5);
	
  \draw[thick] (D1) -- (D5) -- (D4) -- (D2) -- (D3) -- (D1) -- (D4) -- (D3) -- (D6) --
	(D2) -- (D1);
	\draw[thick] (D2) -- (D5) -- (D3);
	\draw[thick] (D4) -- (D6) -- (D1);
  
  \node[draw,circle,inner sep=1.4pt,fill] at (D1) {};
  \node[draw,circle,inner sep=1.4pt,fill] at (D2) {};
  \node[draw,circle,inner sep=1.4pt,fill] at (D3) {};
  \node[draw,circle,inner sep=1.4pt,fill] at (D4) {};
  \node[draw,circle,inner sep=1.4pt,fill] at (D5) {};
  \node[draw,circle,inner sep=1.4pt,fill] at (D6) {};
  \end{tikzpicture}

\end{tabular}
\caption{A $(3,4)$-tight graph $G$ with no spanning $(2,3)$-tight subgraph.}
\label{fig:34eg}
\end{figure}

\subsection{Second counterexample}\label{subsec:2ndcounter}

It is tempting to now posit that the necessary condition taking into account of the ordering of Dilworth truncations gives the correct exact characterisation for the Hadamard product of multiple linear spaces.
For example, when $d=3$, for each $F \subseteq [m]$ we have that
\begin{equation}\label{eq:sharper}
r_{{\cal M}(L_1\ast L_2 \ast L_3)}(F)\leq 
\min_{\sigma \in S_3}
\{((r_{\sigma(1)}+r_{\sigma(2)}-1)^{\rm MD}+ r_{\sigma(3)}-1)^{\rm MD}(F)\},
\end{equation}
where $S_3$ denotes the permutations of $\{1,2,3\}$. Although the value in the right side can be smaller than 
$(r_1+r_2+r_3-2)^{\rm MD}(F)$,
it is plausible to believe that (\ref{eq:sharper}) is actual an equality.
Unfortunately, this is also not true as shown below.

Our second counterexample is the case when each $L_i$ is a linear representation of a \emph{partition matroid}.
Let $E$ be a finite set, and let ${\cal E}=\{E_1,\dots, E_n\}$ be a partition of $E$. By definition, in the partition matroid of ${\cal E}$, $F\subseteq E$ is independent if and only if $|F\cap E_i|\leq 1$ for all $i=1,\dots, n$.
The partition matroid of ${\cal E}$ admits a linear representation over any field by the incidence matrix $I_{{\cal E}}$ of the partition, which is 
a $|E|\times n$-matrix whose $(e,i)$-th entry is 1 if $e\in E$ belongs to $E_i$ for $i\in [n]$ and 0 otherwise.

Suppose we have $d$ partitions
${\cal E}_i=\{E_{i,1}, E_{i,2}, \dots, E_{i,n_i}\}$ of $E$ for $i=1,\dots, d$.
Let $L_i$ be the column space of the incidence matrix $I_{{\cal E}_i}$. Since each of these matrices has all non-zero rows, each $L_i$ is not contained in any coordinate hyperplane.
We now show that $L_1\ast \dots \ast L_d$ is a counterexample of \Cref{con:bernstein} already in the case when $d=3$.

The combinatorics of $L_1\ast \dots \ast L_d$ can be compactly represented if we construct each instance from a \emph{$d$-partite $d$-uniform hypergraph}:
a hypergraph where the vertex set consists of $d$ disjoint sets (called vertex classes)
and each hyperedge is incident to exactly one vertex in each vertex class. 

Let $G=(V,E)$ be a $d$-partite $d$-uniform hypergraph
with vertex classes $V_1, \dots, V_d$.
Let $n_i=|V_i|$ and denote each vertex in $V_i$ by $(i,j)$ for $j\in [n_i]$, i.e., $V_i=\{i\}\times [n_i]$.
Then, for each $i \in [d]$, the hypergraph $G$ induces a partition ${\cal E}_i=\{E_{i,1},\dots, E_{i,n_i}\}$ of $E$ such that 
$
e\in E_{i,j}$ if and only if $e$ is incident to $(i,j)$ in $G$.
See the example below.

As above, let $L_i \subset \CC^E$ be the column space of the incidence matrix $I_{{\cal E}_i}$.
A generic point $x_i\in L_i$ can be written 
as $x_i =I_{{\cal E}_i} p_i$ by taking a generic vector $p_i\in \mathbb{C}^{n_i}$.
Since each $e\in E$ is written as 
$e=\{(1,j_1), (2, j_2), \dots, (d, j_d)\}$ for unique $j_i\in [n_i]$
for each $i=1,\dots, d$, we have
\begin{equation}\label{eq:counter2-1}
x_i(e)=p_i(j_i),
\end{equation}
and hence 
\[
(x_1\ast x_2 \ast \dots \ast x_d)(e)=\prod_{i=1}^d p_i(j_i).
\]
This means that $L_1\ast L_2\ast \dots \ast L_d$
is the projection of the Segre variety of order $n_1 \times \cdots \times n_d$ to $\mathbb{C}^E$.
It is known that the algebraic matroid of the latter variety is represented by the edge-vertex incidence matrix\footnote{For a hypergraph $G=(V,E)$ with $n$ vertices and $m$ edges, the edge-vertex incidence matrix is an $m\times n$ matrix whose $(e,v)$-th entry is 1 if $e$ is incident to $v$ in $G$ and otherwise 0 for $e\in E$ and $v\in V$.} $I_G$ of $G$.
This fact can be also checked easily from \Cref{prop:tangent} and (\ref{eq:counter2-1}).

The following instance is a simple example for which the rank of $I_G$ takes an unexpected value.
Let $d=3$, and let $G=(V_1\cup V_2\cup V_3,E)$ be given by 
\begin{itemize}
\item $n_i=3$ and $V_i=\{i\}\times [n_i] = \{i\}\times [3]$ for $i=1,2,3$, and 
\item  $E=\{\{(1,j_1), (2, j_2), (3, j_3)\}: j_1, j_2, j_3 \text{ are all distinct}\}$.
\end{itemize}
Then, $G$ is the 3-partite 3-uniform hypergraph pictured in \Cref{fig:hypergraph} with nine vertices and six edges.
Under some ordering of the vertices and edges, we see that
\begin{equation*}
    I_G = 
    \left(
    \begin{array}{ccc|ccc|ccc}
        1 & 0 & 0 &  0 & 1 & 0 &  0 & 0 & 1 \\ 
        1 & 0 & 0 &  0 & 0 & 1 &  0 & 1 & 0 \\ 
        0 & 1 & 0 &  1 & 0 & 0 &  0 & 0 & 1 \\ 
        0 & 1 & 0 &  0 & 0 & 1 &  1 & 0 & 0 \\ 
        0 & 0 & 1 &  1 & 0 & 0 &  0 & 1 & 0 \\ 
        0 & 0 & 1 &  0 & 1 & 0 &  1 & 0 & 0
    \end{array}
    \right).
\end{equation*}

\begin{figure}[t]
    \centering
    \includegraphics[width=0.3\linewidth]{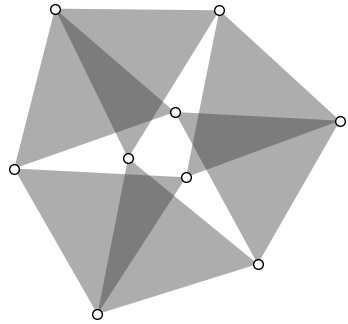}
    \caption{The 3-partite 3-uniform hypergraph described in \Cref{subsec:2ndcounter}.}
    \label{fig:hypergraph}
\end{figure}

In fact, the rows of $I_G$ are linearly dependent.
This can be verified by assigning ${\rm sign}(j_1,j_2,j_3)$ to each edge $e=\{(1,j_1), (2, j_2), (3, j_3)\}\in E$ by regarding 
$(j_1,j_2,j_3)$ as a permutation of $\{1,2,3\}$.
The resulting assignment of $\pm 1$ values forms a nonzero vector in the left kernel of $I_G$ (for the above matrix, this would take the form $(1,-1,-1,1,1,-1)$).
Hence, 
\begin{equation}\label{eq:counter2-3}
{\rm rank}\ {\cal M}(L_1\ast L_2 \ast L_3)\leq 5.
\end{equation}

On the other hand, we now prove that
\begin{equation}\label{eq:counter2-2}
\text{$E$ is independent in the matroid induced by $(r_{\sigma(1)}+r_{\sigma(2)}-1)^{\rm MD}+ r_{\sigma(3)}-1$}
\end{equation}
for any permutation $\sigma \in S_3$.
By symmetry, it suffices to prove  (\ref{eq:counter2-2}) in the case when  $\sigma$ is the identity.
Note that $(r_{1}+r_{2}-1)^{\rm MD}$ is the rank function of the graphic matroid of $G_{12}$,
where $G_{12}$ is the bipartite graph obtained from $G$ by ignoring $V_3$.
By the definition of $G$, $G_{12}$ is actually the cycle graph on the six vertices of $V_1\cup V_2$. Hence, 
\begin{itemize}
    \item for any nonempty proper subset $F$ of $E$, 
    $(r_{1}+r_{2}-1)^{\rm MD}(F)=|F|$ and $r_{3}(F)\geq 1$;
    \item $(r_{1}+r_{2}-1)^{\rm MD}(E)=|E|-1$ and $r_{3}(E)=3$,
\end{itemize}
meaning that 
$|F|\leq (r_{1}+r_{2}-1)^{\rm MD}(F)+ r_{3}(F)-1$ for any nonempty $F\subseteq E$,
and $E$ is independent in the matroid induced by 
$(r_{1}+r_{2}-1)^{\rm MD}+ r_{3}-1$.

By (\ref{eq:counter2-3}) and (\ref{eq:counter2-2}), we conclude that 
the inequality in (\ref{eq:sharper}) is strict for $L_1, L_2, L_3$ defined by $G$.

\section{Hadamard product of generic linear spaces}\label{sec:generic}

We now prove that \Cref{con:bernstein} is true if the linear spaces are chosen to be generic, in the sense of \Cref{subsection:geometric_dilworth}.
The following is a special case of a more general result of Ballico \cite[Theorem 1.2]{BALLICO}.

\begin{theorem}\label{thm:generichadamard}
    Let $L_1, \ldots, L_d \subset \CC^m$ be generic linear spaces with dimensions $n_1,\ldots,n_d$. Then, the algebraic matroid of the Hadamard product $L_1 * \cdots * L_d$ is the rank $k$ uniform matroid on $[m]$,
    where $k = \min \{m , n_1 + \cdots + n_d - (d-1)\}$.
\end{theorem}

To prove this, we require the following results.

\begin{lemma}\label{lem:lowersemi}
    For any positive integers $n_1 \leq \ldots \leq n_d \leq m$,
    the function
    \begin{equation*}
        \phi \colon \prod_{i=1}^d \Gr_{n_i}(\mathbb{C}^m) \rightarrow \mathbb{Z}_{\geq 0},  \quad ( L_1 , \ldots, L_d ) \mapsto \dim L_1 \ast \cdots \ast L_d
    \end{equation*}
    is lower semi-continuous.
\end{lemma}

\begin{proof}
    For any choice of $p=(p_1,\ldots,p_d)$ with $p_i \in \mathbb{C}^{n_i}$ for each $i \in [d]$,
    fix $V_i \subset \mathbb{C}^{m \times n_i}$ to be the nonempty Zariski open subset of matrices of rank $n_i$. Define the map $\psi_p\colon \prod_{i=1}^d V_i \rightarrow \mathbb{C}^{m \times (n_1 + \cdots + n_d)}$ given by
    \[
        (Y_1,\ldots,Y_d) \mapsto \left( \left(\prod_{i \neq 1}{\rm Diag}(Y_ip_i) \right) Y_1 \ \Bigg\vert \ \left(\prod_{i \neq 2}{\rm Diag}(Y_ip_i) \right) Y_2 \ \Bigg\vert \ \cdots \ \Bigg\vert \ \left(\prod_{i \neq d}{\rm Diag}(Y_ip_i) \right) Y_d\right).
    \]
    Since the rank function is lower semi-continuous,
    it is clear that $\rank \psi_p$ is also lower semi-continuous.
    Fix the natural quotient map $\pi \colon \prod_{i=1}^d V_i \rightarrow \prod_{i=1}^d \Gr_{n_i}(\mathbb{C}^m)$.
    If $Y_i \in V_i$ is chosen to be a matrix representative of $L_i$,
    then there exists a nonempty Zariski open set $U \subset \prod_{i=1}^d \mathbb{C}^{n_i}$ such that
    \begin{equation*}
        \phi (L_1,\ldots,L_d) = \rank \psi_p(Y_1,\ldots,Y_d) \qquad \text{ for all } p \in U.
    \end{equation*}
    From this construction, it is now easy to see that $\phi$ is lower semi-continuous. 
\end{proof}

The following lemma provides an inductive construction for linear spaces $L_1,\ldots, L_d \subset \mathbb{C}^m$ where \Cref{con:bernstein} holds.

\begin{lemma}\label{lem:construction}
    For any positive integers $n_1,\ldots, n_d, m$ with $n_i \leq m$ for each $i \in [d]$,
    there exist matrices $Y_1 \in \mathbb{C}^{m \times n_1}, \ldots, Y_d \in \mathbb{C}^{m \times n_d}$ with $\operatorname{rank} Y_i = n_i$ such that the following holds:
    \begin{enumerate}
        \item Each row of each matrix $Y_i$ contains exactly one entry which is 1, and all other entries for the row are 0.
        \item The matrix concatenation $(Y_1 \  \cdots \  Y_d )$ has rank $\min \{ m , n_1 + \cdots + n_d - (d-1)\}$.
    \end{enumerate}
\end{lemma}

\begin{proof}
    Throughout the proof, we fix our choice of $m$.
    If $n_1 = \cdots = n_d = 1$,
    then we set each $Y_i$ to be the unique all-ones $m \times 1$ matrix.
    With this, we see that $(Y_1 \  \cdots \  Y_d )$ has rank $\min \{ m , 1 + \cdots + 1 - (d-1)\} = 1$.
    Now suppose that the statement of the lemma holds with integers $n_1-1, n_2 \ldots ,n_d$ witnessed by the matrices $Y_1', Y_2, \ldots, Y_d$.
    There are now two possible cases.

    \textbf{Case 1:} Suppose that the concatenation $(Y_1' \ Y_2 \  \cdots \  Y_d )$ has rank $m$.
    As $n_1 -1 < m$,
    there exists a column $j$ of $Y_1'$ which contains at least two non-zero entries on rows $s,t \in [m]$.
    We now set $Y_1$ to be the matrix formed from $Y'_1$ by adding a new column $n_1$ with a single 1 in the $s$-th row,
    and replacing the $(j,s)$ entry with a 0.
    From this it is easy to see that $\rank Y_1 = \rank Y'_1 + 1 = n_1$,
    and
    \begin{equation*}
        m = \rank ( Y_1' \ Y_2 \  \cdots \  Y_d ) \leq \rank ( Y_1 \ Y_2 \  \cdots \  Y_d ) \leq m
    \end{equation*}
    as desired.

    \textbf{Case 2:} Suppose that $(Y_1' \ Y_2 \  \cdots \  Y_d )$ has rank strictly less than $m$.
    Then there exists a circuit $C \subset [m]$ of the row matroid of $(Y_1' \ Y_2 \  \cdots \  Y_d )$.
    Since all the entries of $(Y_1' \ Y_2 \  \cdots \  Y_d )$ are 0 or 1, there exists a unique (up to scalar multiplication) vector $\alpha \in \mathbb{R}^m$ with support $C$ contained in the left kernel of $(Y_1' \ Y_2 \  \cdots \  Y_d )$.
    Moreover, $\alpha$ must have both positive and negative entries (since all entries of $(Y_1' \ Y_2 \  \cdots \  Y_d )$ are non-negative).
    Set $C_+ := \{ e \in C : \alpha(e) >0 \}$ and $C_- := \{ e \in C : \alpha(e) < 0 \}$.

    Choose any $a \in C_+$ and set $c$ to be the unique column of $Y'_1$ which has a 1 at position $a$.
    Since the dot product of $\alpha$ with the $c$-th column of $Y'_1$ must sum to 0,
    there exists $b \in C_-$ such that $Y'_1(b,c) = 1$ also.
    We now set $Y_1 \in \mathbb{C}^{m \times n_1}$ to be the matrix formed from $Y'_1$ by adding a new column $n_1$ such that the following holds for each entry $(i,j)$:
    \begin{equation*}
        Y_1 (i,j) := 
        \begin{cases}
            Y_1'(j,c) &\text{if } j=n_1, ~ i \in C_+, \\
            0 &\text{if } j=n_1, ~ i \notin C_+, \\
            0 &\text{if } j=c, ~ i \in C_+, \\
            Y_1'(i,j) &\text{otherwise}
        \end{cases}
    \end{equation*}
    It is clear that $\rank Y_1 = \rank Y_1' + 1 = n_1$.
    We also observe that the left kernel of $(Y_1 \ Y_2 \ \cdots \ Y_d)$ is contained in the left kernel of $(Y_1' \ Y_2 \ \cdots \ Y_d)$ but the former does not contain the vector $\alpha$.
    It hence follows that
    \begin{equation*}
        \rank (Y_1 \ Y_2 \ \cdots \ Y_d) \geq \rank (Y'_1 \ Y_2 \ \cdots \ Y_d) + 1 =  n_1 + \cdots + n_d - (d-1).
    \end{equation*}
    As adding the $n_1$-th column of $Y_1$ to $Y'_1$ produces a matrix of the form $(Y'_1 \ v)$ for some column vector $v$,
    we have
    \begin{equation*}
        \rank (Y_1 \ Y_2 \ \cdots \ Y_d)\leq \rank (Y'_1 \ Y_2 \ \cdots \ Y_d) + 1 =  n_1 + \cdots + n_d - (d-1).
    \end{equation*}
    This now provides the desired equality.
\end{proof}

\begin{proof}[Proof of \Cref{thm:generichadamard}]
    Fix $Y_1,\ldots,Y_d$ to be the matrices described in \Cref{lem:construction} with corresponding column spaces $L_1,\ldots,L_d$.
    The algebraic matroid of $L_1 \ast \cdots \ast L_d$ is the row matroid of $(Y_1 \ \cdots \ Y_d)$;
    see \Cref{subsec:2ndcounter} for more details.
    Hence $\dim (L_1 \ast \cdots \ast L_d) = k$.
    It now follows from \Cref{lem:lowersemi} that this property holds for any generic linear spaces $L_1,\ldots,L_d$.
    As coordinate projections of generic linear spaces are generic linear spaces,
    it follows that $\mathcal{M}(L_1 \ast \cdots \ast L_d)$ is the rank $k$ uniform matroid on $[m]$.
\end{proof}

\section{Amoebas of complex linear spaces}\label{sec:amoeba}

Throughout this section we denote the complex conjugate of a vector $v$ (respectively, linear space $L$) by $\overline{v}$ (respectively, $\overline{L}$), and we denote the imaginary unit of $\mathbb{C}$ by $\mathbf{i}$.

Given $\mathbb{C}^* := \mathbb{C} \setminus \{0\}$,
we define the continuous map
\begin{equation*}
    \Log \colon (\mathbb{C}^*)^m \rightarrow \mathbb{R}^m , ~ (z_1,\ldots,z_m) \mapsto (\log |z_1|, \ldots \log |z_m|).
\end{equation*}

\begin{definition}
    Let $V \subset \mathbb{C}^m$ be an irreducible algebraic variety not contained in a coordinate hyperplane.
    The \emph{amoeba} $\mathcal{A}(V)$ of $V$ is the closure of $\Log(V \cap (\mathbb{C}^*)^m)$ in the Euclidean topology.    
\end{definition}

In this section we provide an alternative proof for the following result.

\begin{theorem}[Draisma et al.~\cite{amoeba}]\label{t:amoeba}
    Let $V \subset \mathbb{C}^m$ be a linear space not contained in a coordinate hyperplane. Denote by $r_{{\cal M}(V)}$ the rank function of the algebraic matroid of $V$. Then:
    \[
        \dim \mathcal{A}(V)  = \min\left\{
        \sum_{i=1}^t(2 r_{\mathcal{M}(V)}(E_i) - 1) : \text{$\{E_1,\dots, E_t\}$ is a partition of $[m]$}
        \right\}.
    \]
\end{theorem}

The key step of our proof is to express the (real) dimension of the amoeba ${\cal A}(V)$ as the (complex) dimension of the Hadamard product $V * \overline{V}$. Before starting with the proof, we first require the following easy result.

\begin{lemma}\label{lem:technlem}
    Let $Z = \{(q,\overline{q}) : q \in \mathbb{C}^n\}$.
    Then $Z$ is Zariski-dense in $\mathbb{C}^{2n}$.
\end{lemma}

\begin{proof}
    Let $f \in \mathbb{C}[\mathbf{x}, \mathbf{y}]$ a polynomial such that $f(\mathbf{z},\overline{\mathbf{z}}) = 0$ for all $\mathbf{z} \in \mathbb{C}^n$. Hence, $f(\mathbf{x} + \mathbf{i} \mathbf{y}, \mathbf{x} - \mathbf{i} \mathbf{y}) = 0$ for all $\mathbf{x}, \mathbf{y} \in \mathbb{R}^n$. Define the complex polynomial $g(\mathbf{x},\mathbf{y}) = f(\mathbf{x} + \mathbf{i} \mathbf{y}, \mathbf{x} - \mathbf{i} \mathbf{y}) \in \mathbb{C}[\mathbf{x},\mathbf{y}]$. Since $f(\mathbf{x}, \mathbf{y}) = g\bigl( \frac{\mathbf{x + \mathbf{y}}}{2}, \frac{\mathbf{x} - \mathbf{y}}{2\mathbf{i}}\bigr)$, we have that $f = 0$ if and only if $g = 0$. Now, write $g = g_1 + i g_2$, where $g_1, g_2 \in \mathbb{R}[\mathbf{x},\mathbf{y}]$. Since $g(\mathbf{x, \mathbf{y}}) = 0$ for all $\mathbf{x}, \mathbf{y} \in \mathbb{R}^n$, we have that $g_1(\mathbf{x},\mathbf{y}) = g_2(\mathbf{x},\mathbf{y}) = 0$ for all $\mathbf{x}, \mathbf{y} \in \mathbb{R}^n$. Hence, $g_1 = g_2 = 0$ as polynomials. It follows that $g = 0$, which implies that $f = 0$.
\end{proof}

\begin{proof}[Proof of \Cref{t:amoeba}]
To analyse the image of $\Log$, we want to consider it as a map between real manifolds.
To do this,
we define the following maps: 
\begin{gather*}
    \phi\colon (\mathbb{R}^2\setminus \{(0,0)\})^{m}\rightarrow (\mathbb{C}^*)^m, ~ \big( a_i, b_i  \big)_{i \in [m]} \mapsto  (a_i + b_i \mathbf{i})_{i \in [m]},\\
    \Log_{\mathbb{R}} \colon (\mathbb{R}^2\setminus \{(0,0)\})^{m} \rightarrow \mathbb{R}^m , ~ \big(a_i, b_i  \big)_{i \in [m]} \mapsto \frac{1}{2}\left(\log (a_1^2 + b_1^2), \ldots, \log (a_m^2 + b_m^2)\right).
\end{gather*}
We note here that $\Log = \Log_{\mathbb{R}} \circ \phi^{-1}$.
The map $\Log_{\mathbb{R}}$ is also a real analytic map, whose Jacobian at  $p = \big( (a_i)_{i \in [m]}, ( b_i )_{i \in [m]} \big)$ is 
\begin{equation*}
    J_p (\Log_{\mathbb{R}}) = \Bigg(\diag \left(\left( \frac{2a_i}{a_i^2 + b_i^2} \right)_{i \in [m]}\right) \ \Bigg\vert \ \diag \left(\left( \frac{2b_i}{a_i^2 + b_i^2} \right)_{i \in [m]}\right) \Bigg).
\end{equation*}
Now take $V \subset \mathbb{C}^m$ to be a linear space not contained in a coordinate hyperplane.
Let $M \in \mathbb{C}^{m \times n}$ be a matrix with column span $V$, where $n = \dim V$.
Choose a generic point $z \in V \cap (\mathbb{C}^*)^m$;
equivalently, choose a generic point $z' \in \mathbb{C}^{\rm d}$ and then fix $z = Mz'$.
Now set $p = (a,b) =  \phi^{-1}(z)$.
We observe here that if $M = A +\mathbf{i} B$ for real matrices $A, B$ and $z' = v + \mathbf{i} w$ for real vectors $v,w$,
then $a = A v - B w$ and $b = B v + A w$.
With this set-up,
$\mathcal{A}(V)$ can be described as the Euclidean closure of the image of the real analytic function
\begin{equation*}
    (\mathbb{R}^n)^2 \dashrightarrow \mathbb{R}^m, ~ (\tilde{v} ,\tilde{w}) \mapsto \Log_{\mathbb{R}}(A\tilde{v} - B\tilde{w},B\tilde{v}+ A\tilde{w}).
\end{equation*}
Hence,
\begin{align*}
    \dim \mathcal{A}(V) &= \rank \left(J_p (\Log_{\mathbb{R}})
    \begin{pmatrix}
        A & -B \\
        B & A
    \end{pmatrix}
    \right)\\
    &= \rank \left( 
    \begin{pmatrix}
        \diag (a) & \Big\vert &  \diag (b)
    \end{pmatrix}
    \begin{pmatrix}
        A & -B \\
        B & A
    \end{pmatrix}
    \right)\\
    &= \rank 
    \begin{pmatrix}
        \diag (a) A + \diag (b) B  & \Big\vert & -\diag (a) B + \diag (b) A
    \end{pmatrix}
    \\
    &= \rank 
    \begin{pmatrix}
        \diag (\overline{z}) M + \diag (z) \overline{M}  & \Big\vert& \mathbf{i} \diag (\overline{z}) M - \mathbf{i}  \diag (z) \overline{M}
    \end{pmatrix}
    \\
    &= \rank
    \begin{pmatrix}
    \text{Re}(\diag(\overline{z})M) & \Big\vert & \text{Im}(\diag(\overline{z})M)
    \end{pmatrix}
    \\
    &= \rank
    \begin{pmatrix}
    \text{Re}(\diag(\overline{z})M) + \mathbf{i} \ \text{Im}(\diag(\overline{z})M) & \Big\vert & \text{Re}(\diag(\overline{z})M) - \mathbf{i} \ \text{Im}(\diag(\overline{z})M)
    \end{pmatrix}
    \\
    &= \rank
    \begin{pmatrix}
        \diag (\overline{z}) M & \Big\vert & \diag (z) \overline{M} 
    \end{pmatrix}
    \\
    &= \rank
    \begin{pmatrix}
        \diag \left(\overline{M} \  \overline{z'} \right) M & \Big\vert & \diag \left(M  z' \right) \overline{M} 
    \end{pmatrix}
\end{align*}
Now, for each pair $q_1,q_2 \in \mathbb{C}^n$, define the matrix $D(q_1,q_2)$ by
\begin{equation*}
    D(q_1,q_2) := 
    \begin{pmatrix}
        \diag \left(\overline{M} q_2 \right) M & \Big\vert & \diag \left(M q_1 \right) \overline{M} 
    \end{pmatrix}
\end{equation*}
By \Cref{lem:technlem},
for any generic $r \in \mathbb{C}^n$ we have
\begin{equation*}
    \rank D (r,\overline{r}) = \max_{q_1,q_2 \in \mathbb{C}^{\rm D}} \rank D(q_1,q_2) .
\end{equation*}
Since $D(p_1,p_2)$ is the Jacobian of the Hadamard map $\CC^n \times \CC^n \to V \times \overline{V} \rightarrow V \ast \overline{V}$ at a point $(p_1,p_2)$,
we have that
\begin{equation*}
    \dim \mathcal{A}(V) = \dim V \ast \overline{V}.
\end{equation*}
As $\mathcal{M}(\overline{V}) = \mathcal{M}(V)$, it follows from \Cref{thm:bernstein} that 
\begin{equation*}
    \dim \mathcal{A}(V) = \dim V \ast \overline{V} = \rank \mathcal{M}_{r_{\mathcal{M}(V)} + r_{\mathcal{M}(\overline{V})} - 1} 
    = \rank \mathcal{M}_{ 2r_{\mathcal{M}(V)} - 1}.
\end{equation*}
The rank formula of the Dilworth truncation of $2r_{\mathcal{M}(V)} - 1$ now gives the result.
\end{proof}

\subsection*{Acknowledgements}

This material is based upon work supported by the National Science Foundation under Grant No.\ DMS-1929284 while the authors were in residence at the Institute for Computational and Experimental Research in Mathematics in Providence, RI, during the \textit{Geometry of Materials, Packings and Rigid Frameworks} semester program. 
The authors would like to thank Meera Sitharam and Louis Theran for their discussions on Dilworth truncations during the program.
D.\,A.\ is a member of GNSAGA (INdAM).
S.\,D.\ was supported by the Heilbronn Institute for Mathematical Research. 
S.\,T.\ was supported by JSPS KAKENHI Grant Number 24K21315 and 24K14835. 
D.\,A.\ is grateful to Sarah Eggleston for initiating the discussion on amoebas of linear spaces.
S.\,T.\ would like to thank Bill Jackson for a stimulating discussion on \Cref{thm:dil1}.

\bibliographystyle{plainurl}
\bibliography{ref}

\begin{thebibliography}{10}

\bibitem{BALLICO}
Edoardo Ballico.
\newblock Projective surfaces not as {H}adamard products and the dimensions of the {H}adamard joins.
\newblock {\em Journal of Pure and Applied Algebra}, 229(1):107853, 2025.
\newblock \href {https://doi.org/10.1016/j.jpaa.2024.107853} {\path{doi:10.1016/j.jpaa.2024.107853}}.

\bibitem{bernstein22}
Daniel~Irving Bernstein.
\newblock Generic symmetry-forced infinitesimal rigidity: Translations and rotations.
\newblock {\em SIAM Journal on Applied Algebra and Geometry}, 6(2):190--215, 2022.
\newblock \href {https://doi.org/10.1137/20M1346961} {\path{doi:10.1137/20M1346961}}.

\bibitem{bocci2024hadamard}
Cristiano Bocci and Enrico Carlini.
\newblock {\em Hadamard Products of Projective Varieties}.
\newblock Springer, 2024.
\newblock \href {https://doi.org/10.1007/978-3-031-54263-3} {\path{doi:10.1007/978-3-031-54263-3}}.

\bibitem{BCK17}
Cristiano Bocci, Enrico Carlini, and Joe Kileel.
\newblock Hadamard products of linear spaces.
\newblock {\em Journal of Algebra}, 448:595--617, 2016.
\newblock \href {https://doi.org/10.1016/j.jalgebra.2015.10.008} {\path{doi:10.1016/j.jalgebra.2015.10.008}}.

\bibitem{capco}
Jose Capco, Matteo Gallet, Georg Grasegger, Christoph Koutschan, Niels Lubbes, and Josef Schicho.
\newblock The number of realizations of a {L}aman graph.
\newblock {\em SIAM Journal on Applied Algebra and Geometry}, 2(1):94--125, 2018.
\newblock \href {https://doi.org/10.1137/17M1118312} {\path{doi:10.1137/17M1118312}}.

\bibitem{COULLARD1991223}
Collette~R. Coullard, John~G. {del Greco}, and Donald~K. Wagner.
\newblock Representations of bicircular matroids.
\newblock {\em Discrete Applied Mathematics}, 32(3):223--240, 1991.
\newblock \href {https://doi.org/10.1016/0166-218X(91)90001-D} {\path{doi:10.1016/0166-218X(91)90001-D}}.

\bibitem{CMS10:GeometryRBM}
Mar{\'\i}a~Ang{\'e}lica Cueto, Jason Morton, and Bernd Sturmfels.
\newblock Geometry of the restricted {B}oltzmann machine.
\newblock {\em Algebraic Methods in Statistics and Probability}, 516:135--153, 2010.
\newblock \href {https://doi.org/10.1090/conm/516/10172} {\path{doi:10.1090/conm/516/10172}}.

\bibitem{CTY10:Implicitization}
Mar{\'\i}a~Ang{\'e}lica Cueto, Enrique~A. Tobis, and Josephine Yu.
\newblock An implicitization challenge for binary factor analysis.
\newblock {\em Journal of Symbolic Computation}, 45(12):1296--1315, 2010.
\newblock \href {https://doi.org/10.1016/j.jsc.2010.06.011} {\path{doi:10.1016/j.jsc.2010.06.011}}.

\bibitem{dewar2023rigidgraphscylindricalnormed}
Sean Dewar and Derek Kitson.
\newblock Rigid graphs in cylindrical normed spaces.
\newblock {\em SIAM Journal on Discrete Mathematics}, 39(3):1545--1567, 2025.
\newblock \href {https://doi.org/10.1137/23M1587543} {\path{doi:10.1137/23M1587543}}.

\bibitem{amoeba}
Jan Draisma, Sarah Eggleston, Rudi Pendavingh, Johannes Rau, and Chi~Ho Yuen.
\newblock The amoeba dimension of a linear space.
\newblock {\em Proceedings of the American Mathematical Society}, 152(6):2385--2401, 2024.
\newblock \href {https://doi.org/10.1090/proc/16744} {\path{doi:10.1090/proc/16744}}.

\bibitem{dunstan1976matroids}
Frank D.~J. Dunstan.
\newblock Matroids and submodular functions.
\newblock {\em The Quarterly Journal of Mathematics}, 27(3):339--348, 1976.
\newblock \href {https://doi.org/10.1093/qmath/27.3.339} {\path{doi:10.1093/qmath/27.3.339}}.

\bibitem{edmonds1970submodular}
Jack Edmonds.
\newblock Submodular functions, matroids and certain polyhedra.
\newblock In R.~Guy, H.~Hanani, N.~Sauer, and J.~Schönheim, editors, {\em Combinatorial Structures and Their Applications}, pages 69--87. Gordon and Breach, New York, 1970.
\newblock Proceedings of Calgary International Conf.\ (1969).

\bibitem{edmonds1966submodular}
Jack Edmonds and Gian-Carlo Rota.
\newblock Submodular set functions.
\newblock In {\em Proceedings of the Waterloo Combinatorics Conference}, 1966.
\newblock Abstract.

\bibitem{jackson}
Bill Jackson and John~C. Owen.
\newblock A characterisation of the generic rigidity of 2-dimensional point–line frameworks.
\newblock {\em Journal of Combinatorial Theory, Series B}, 119:96--121, 2016.
\newblock \href {https://doi.org/10.1016/j.jctb.2015.12.007} {\path{doi:10.1016/j.jctb.2015.12.007}}.

\bibitem{katoh2009infinitesimal}
Naoki Katoh and Shin-ichi Tanigawa.
\newblock On the infinitesimal rigidity of bar-and-slider frameworks.
\newblock In Yingfei Dong, Ding~Zhu Du, and Oscar Ibarra, editors, {\em Proceedings of the 20th International Symposium on Algorithms and Computation (ISAAC 2009)}, volume 5878 of {\em Lecture Notes in Computer Science}, pages 524--533, Honolulu, Hawaii, USA, 2009.
\newblock \href {https://doi.org/10.1007/978-3-642-10631-6_54} {\path{doi:10.1007/978-3-642-10631-6_54}}.

\bibitem{KITSON2020}
Derek Kitson and Rupert~H. Levene.
\newblock Graph rigidity for unitarily invariant matrix norms.
\newblock {\em Journal of Mathematical Analysis and Applications}, 491(2):124353, 2020.
\newblock \href {https://doi.org/10.1016/j.jmaa.2020.124353} {\path{doi:10.1016/j.jmaa.2020.124353}}.

\bibitem{Laman1970}
Gerard Laman.
\newblock On graphs and rigidity of plane skeletal structures.
\newblock {\em Journal of Engineering Mathematics}, 4:331--340, 1970.
\newblock \href {https://doi.org/10.1007/BF01534980} {\path{doi:10.1007/BF01534980}}.

\bibitem{lovasz1977flats}
L{\'a}szl{\'o'} Lov{\'a}sz.
\newblock Flats in matroids and geometric graphs.
\newblock In P.~J. Cameron, editor, {\em Surveys in Combinatorics}, pages 363--374. Academic Press, London, 1977.

\bibitem{lovasz82}
L{\'a}szl{\'o'} Lov{\'a}sz and Yechiam Yemini.
\newblock On generic rigidity in the plane.
\newblock {\em SIAM Journal on Algebraic Discrete Methods}, 3(1):91--98, 1982.
\newblock \href {https://doi.org/10.1137/0603009} {\path{doi:10.1137/0603009}}.

\bibitem{Mon16}
Guido Mont{\'u}far.
\newblock Restricted {B}oltzmann machines: {I}ntroduction and review.
\newblock In {\em Information Geometry and its Applications IV}, pages 75--115. Springer, 2016.
\newblock \href {https://doi.org/10.1007/978-3-319-97798-0_4} {\path{doi:10.1007/978-3-319-97798-0_4}}.

\bibitem{MM17:DimensionKronecker}
Guido Mont{\'u}far and Jason Morton.
\newblock Dimension of marginals of {K}ronecker product models.
\newblock {\em SIAM Journal on Applied Algebra and Geometry}, 1(1):126--151, 2017.
\newblock \href {https://doi.org/10.1137/16M1077489} {\path{doi:10.1137/16M1077489}}.

\bibitem{oxley}
James Oxley.
\newblock {\em Matroid Theory}.
\newblock Oxford University Press, 2011.
\newblock \href {https://doi.org/10.1093/acprof:oso/9780198566946.001.0001} {\path{doi:10.1093/acprof:oso/9780198566946.001.0001}}.

\bibitem{pollaczek-geiringer1927}
Hilda Pollaczek-Geiringer.
\newblock {Über die Gliederung ebener Fachwerke}.
\newblock {\em ZAMM - Journal of Applied Mathematics and Mechanics / Zeitschrift für Angewandte Mathematik und Mechanik}, 7(1):58--72, 1927.
\newblock \href {https://doi.org/10.1002/zamm.19270070107} {\path{doi:10.1002/zamm.19270070107}}.

\bibitem{Raz2019}
Orit~E. Raz and Avi Wigderson.
\newblock {\em Subspace Arrangements, Graph Rigidity and Derandomization Through Submodular Optimization}, pages 377--415.
\newblock Springer Berlin Heidelberg, Berlin, Heidelberg, 2019.
\newblock \href {https://doi.org/10.1007/978-3-662-59204-5_12} {\path{doi:10.1007/978-3-662-59204-5_12}}.

\bibitem{rosen2020algebraic}
Zvi Rosen, Jessica Sidman, and Louis Theran.
\newblock Algebraic matroids in action.
\newblock {\em The American Mathematical Monthly}, 127(3):199--216, 2020.
\newblock \href {https://doi.org/10.1080/00029890.2020.1689781} {\path{doi:10.1080/00029890.2020.1689781}}.

\bibitem{rosen2025linearizing}
Zvi Rosen, Jessica Sidman, and Louis Theran.
\newblock Linearizing algebraic matroids.
\newblock {\em arXiv preprint}, 2025.
\newblock \href {https://arxiv.org/abs/2507.07220} {\path{arXiv:2507.07220}}.

\bibitem{schrijver2003combinatorial}
Alexander Schrijver.
\newblock {\em Combinatorial Optimization: Polyhedra and Efficiency}, volume~24 of {\em Algorithms and Combinatorics}.
\newblock Springer, 2003.
\newblock URL: \url{https://link.springer.com/book/9783540443896}.

\bibitem{tangiawa}
Shin-ichi Tanigawa.
\newblock Matroids of gain graphs in applied discrete geometry.
\newblock {\em Transactions of the American Mathematical Society}, 367(12):8597--8641, 2015.
\newblock \href {https://doi.org/10.1090/tran/6401} {\path{doi:10.1090/tran/6401}}.

\bibitem{whiteley}
Walter Whiteley.
\newblock A matroid on hypergraphs, with applications in scene analysis and geometry.
\newblock {\em Discrete {\&} Computational Geometry}, 4(1):75--95, 1989.
\newblock \href {https://doi.org/10.1007/BF02187716} {\path{doi:10.1007/BF02187716}}.

\bibitem{whiteley1996some}
Walter Whiteley.
\newblock Some matroids from discrete applied geometry.
\newblock In James~E. Bonin, Joseph~G. Oxley, and Bruce Servatius, editors, {\em Matroid Theory (Seattle, WA, 1995)}, volume 197 of {\em Contemporary Mathematics}, pages 171--312. American Mathematical Society, Providence, RI, 1996.
\newblock \href {https://doi.org/10.1090/conm/197/02540} {\path{doi:10.1090/conm/197/02540}}.

\end{thebibliography}

\appendix

\section{Proof of \texorpdfstring{\Cref{thm:dil1}}{geometric Dilworth truncation}}\label{appendix}
\newcommand{\MA}{{\mathcal A}}
\newcommand{\MI}{{\mathcal I}}
\newcommand{\MF}{{\mathcal F}}
\newcommand{\MG}{{\mathcal G}}
\newcommand{\MC}{{\mathcal C}}
\newcommand{\MB}{{\mathcal B}}
\newcommand{\MH}{{\mathcal H}}
\newcommand{\MW}{{\mathcal W}}
\newcommand{\MX}{{\mathcal X}}
\newcommand{\MY}{{\mathcal Y}}
\newcommand{\D}{{\mathcal D}}
\newcommand{\MZ}{{\mathcal Z}}
\newcommand{\ML}{{\mathcal L}}
\newcommand{\MU}{{\mathcal U}}
\newcommand{\MM}{{\mathcal M}}

To prove \Cref{thm:dil1}, we instead prove a slight strengthening of the result.
We first require the following definition.

\begin{definition}
Let $\mathbb{F}\in \{\mathbb{R},\mathbb{C}\}$.
Let $E$ be a finite set and $\MA=\{A_e:e\in E\}$ be a family of subspaces of $\mathbb{F}^n$.  
We say that a codimension one linear subspace $H$ of $\mathbb{F}^n$ is {\em regular} with respect to ${\cal A}$ if 
\begin{itemize}
\item $H$ intersects ${\cal A}$ transversally, i.e., $\dim A_e\cap H=\dim A_e-1$ for all $e\in E$, and
\item for each $F\subseteq E$, $\dim \langle A_e\cap H': e\in F\rangle$ is maximized by $H'=H$ over all 
codimension one linear subspace $H'$ intersecting ${\cal A}$ transversally.
\end{itemize}
\end{definition}

Note that $H$ is regular if $H$ is chosen as follows.
We first choose  bases $B_e$ of $A_e$ for $e\in E$
and then choose $H$ such that the set of coordinates of a normal vector to $H$ are algebraically independent over the extension field obtained by adding the coordinates of the vectors in $\bigcup_{e\in E} B_e$ to $\QQ$.
Alternatively, any generic codimension one subspace $\mathbb{F}^n$ is regular.
Because of this, \Cref{thm:dil1} follows from the following result. From now on, given a family ${\cal A}=\{A_e: e\in E\}$ and $F\subseteq E$,
we denote ${\cal A}_F=\langle A_e: e\in F \rangle$ and $\langle {\cal A} \rangle = {\cal A}_E$.

\begin{theorem}\label{thm:dilregular}
Let $\mathbb{F}\in \{\mathbb{R}, \mathbb{C} \}$.
Suppose $\MA= \{A_e:e\in E\}$ is a finite family of non-trivial subspaces of $\FF^n$,  and $H$ is a codimension one linear subspace which is regular  with respect to $\MA$. 
Then
\begin{equation*}
\dim 
\langle \MA\rangle=\min\left\{\sum_{i=1}^t (\dim\langle A_e : e \in F_i \rangle-1):\{F_1,\ldots,F_t\} \mbox{ is a partition of }E\right\}.
\end{equation*}
\end{theorem}

To prove \Cref{thm:dilregular}, we first solve the following special case. We say that a family ${\cal A} = \{A_e : e \in E\}$ is {\em connected} if $\MA_{F'} \cap  \MA_{F''} \neq \{0\}$ 
for all partitions $\{F',F''\}$ of $E$. 

\begin{theorem}\label{thm:connected}
Given $\mathbb{F}\in \{\mathbb{R}, \mathbb{C}\}$, let $\MA= \{A_e:e\in E\}$ be a finite family of non-trivial subspaces of $\FF^n$,  and $H$ be a regular codimension one subspace with respect to $\MA$. 
Suppose $\{A_e\cap H:e\in E\}$ is connected. Then 
$$\dim \langle A_e\cap H: e\in E\rangle=\dim \langle {\cal A}\rangle-1.$$
\end{theorem}
\begin{proof}
We proceed by induction on $|E|$.
The case when $|E|=1$ follows from the first property of the regularity of $H$, so we may assume that $|E|\geq 2$.
Let $B_e=A_e\cap H$ for each $e\in E$ and set $\MB=\{B_e: e\in E\}$. 
Since $H$ is a codimension one subspace intersecting $A_e$ transversally, for each $e\in E$,
\begin{equation}\label{eq:0n}
\dim  \langle \MB \cup \{A_e\}\rangle =\dim \langle \MB\rangle+1.
\end{equation}

\begin{claim}\label{clm:generic}
$A_{e_2}\subseteq \langle \MB\cup\{ A_{e_1}\}\rangle$ holds for some distinct $e_1, e_2\in E$.
\end{claim}
\begin{proof}
Suppose, for a contradiction, that 
\begin{equation}\label{eq:1n}
\text{
 $A_{e_2}\not \subseteq \langle \MB\cup\{ A_{e_1}\}\rangle$
 for all distinct $e_1, e_2\in E$.}
\end{equation}
 Choose a basis $X_e$ of $B_e$ for each $e\in E$.
Since $\MB$ is connected, there is 
a circuit (i.e.~a minimally linearly dependent set) 
$C \subseteq \bigcup_{e\in E} X_e$ such that $C\cap X_{e_1}\neq \emptyset$ and 
 $C\cap X_{e_2}\neq \emptyset$ for two distinct $e_1, e_2\in E$.
Indeed, pick any partition $\{E_1, E_2\}$ of $E$.
The connectivity of ${\cal B}$ implies that 
\begin{equation*}
    \left\langle \bigcup_{e\in E_1} X_e \right\rangle\cap  \left\langle \bigcup_{e\in E_2} X_e\right\rangle\neq \{0\}.
\end{equation*}
Pick any nonzero vector $x$ in this intersection,
and let $C_i$ be a circuit in $\{x\}\cup \bigcup_{e\in E_i} X_e$  for each $i=1,2$.
Then, the circuit elimination axiom implies that 
$(C_1\cup C_2)\setminus \{x\}$ contains a circuit $C$, satisfying the desired property.

 Choose $x_i\in C\cap X_{e_i}$
 and $y_i\in A_{e_i}\setminus B_{e_i}$ for each $i=1,2$.
 Let $Z$ be a basis of $\langle\MB \rangle$
with $C\setminus \{x_2\}\subseteq Z\subseteq \bigcup_{e\in E}X_e$.
 By 
(\ref{eq:1n}), $y_2\notin \langle Z\cup\{y_1\}\rangle$.
 Also, $y_1\notin \langle Z\rangle$ since $y_1\notin H$ and $Z\subset H$.
 Hence, 
 \[
 \dim\langle Z\cup\{y_1,y_2\}\rangle =\dim\langle Z\rangle +2 = |Z| +2.
 \]
 In addition,
 since $C$ is the unique circuit in $Z\cup\{y_1,y_2,x_2\}$  and $x_1\in C$, we have
 $x_2\notin \langle (Z\setminus\{x_1\})\cup\{y_1,y_2\}\rangle$.
 Thus $(Z\setminus\{x_1\})\cup\{y_1,y_2\}$ satisfies:
 \[
  x_1, x_2\notin \langle (Z\setminus\{x_1\})\cup\{y_1,y_2\}\rangle \quad \text{ and } \quad
   \dim\langle (Z\setminus\{x_1\})\cup\{y_1,y_2\}\rangle =\dim\langle Z\rangle+1=\dim \langle \MB \rangle +1.
  \]

By (\ref{eq:0n}), we have  $\dim \langle \MB \rangle \leq \dim \langle \MA \rangle-1$. If equality holds then we are done for \Cref{thm:connected}, so we may assume that $\dim \langle \MB \rangle\leq \dim \MA-2$.
We now consider perturbing $H$ such that the perturbed codimension one subspace $H_{\varepsilon}$ contains 
 $(Z\setminus \{x_1\})\cup \{
 (1-\varepsilon)x_1+\varepsilon y_1,  (1-\varepsilon)x_2+\varepsilon y_2\}$ for $\varepsilon \in \mathbb{R}$.
Since $Z\subset \bigcup_{e\in E}A_e$ and $(1-\varepsilon)x_i+\varepsilon y_i\in A_{e_i}$ for $1\leq i\leq 2$, we have
$$(Z\setminus \{x_1\})\cup \{
 (1-\varepsilon)x_1+\varepsilon y_1,  (1-\varepsilon)x_2+\varepsilon y_2\}
\subset \langle A_e\cap H_{\varepsilon}: e\in E\rangle \qquad \text{ for all } \epsilon \in \mathbb{R}.$$
Since $Z\setminus\{x_1\}\cup\{y_1,y_2\}$ is linearly independent, 
$(Z\setminus \{x_1\})\cup \{
 (1-\varepsilon)x_1+\varepsilon y_1,  (1-\varepsilon)x_2+\varepsilon y_2\}$ is linearly independent 
 for almost all values of $\varepsilon$.
Hence, for almost all values of $\varepsilon$, we have
\begin{equation}\label{eq:2n}
\dim \langle A_e\cap H_{\varepsilon}: e\in E\rangle\geq \dim \langle\MB\rangle+1>\dim \langle A_e\cap H: e\in E\rangle.
\end{equation}
 In addition,  $H_{\varepsilon}$ intersects $\MA$ transversally when $\varepsilon$ is sufficiently small.
 However, (\ref{eq:2n}) 
 contradicts the second property of the regularity of $H$.
\end{proof}

By \Cref{clm:generic}, we can choose distinct $e_1, e_2\in E$ with 
$A_{e_2}\subseteq \langle\MB\cup \{A_{e_1}\}\rangle$.
Let $A_{e'}=A_{e_1}+A_{e_2}$ for a new index $e'$,
and let ${\cal A}'=\{A_e:e\in E\setminus \{e_1,e_2\}\}\cup \{A_{e'}\}$.
Since $A_{e'}$ is not contained in $H$,
$H$ intersects ${\cal A}'$ transversally.
Also, the second property of the regularity of $H$ with respect to $\MA$ implies that with respect to $\MA'$.
So, $H$ is regular with respect to $\MA'$.
The connectivity of $\MA$ implies the connectivity of $\MA'$.
Thus, we can apply induction to $\MA'$ to obtain
\begin{equation*}
    \dim \langle B_e: e\in E\setminus \{e_1,e_2\}\cup \{e'\}\rangle=\dim \langle \MA'\rangle-1=\dim \langle \MA\rangle-1,    
\end{equation*}
where $B_{e'}=A_{e'}\cap H$.
Also, 
since $A_{e_2}\subseteq \langle \MB\cup\{A_{e_1}\}\rangle$,
we have $A_{e'} \subseteq \langle \MB\cup\{A_{e_1}\}\rangle$.
This now implies that $B_{e'}  \subseteq \langle \MB\rangle$,
and so $\langle B_e: e\in E\setminus \{e_1,e_2\}\cup \{e'\}\rangle=\langle B_e: e\in E\rangle$.
Thus, we obtain 
$\dim \langle B_e: e\in E\rangle=\dim \langle \MA\rangle -1$.
This completes the proof.
\end{proof}

\begin{proof}[Proof of \Cref{thm:dilregular}] 
Denote $B_e=A_e\cap H$ and $\MB=\{B_e: e\in E\}$.
Pick any  partition $\{F_1,\ldots,F_t\}$ of $E$.
Then  
\begin{equation}\label{eq:dil1}
\dim \langle \MB\rangle 
\leq \sum_{i=1}^t \dim \MB_{F_i}
\leq \sum_{i=1}^t (\dim \MA_{F_i}-1),
\end{equation}
where the last equality follows from the transversality of $H$.

We complete the proof by showing that equality in (\ref{eq:dil1}) does hold  for some partition of $E$. We ensure that equality holds in the first inequality in (\ref{eq:dil1}) by choosing a partition $\{F_1,\ldots,F_t\}$ of $E$ such that $\langle \MB\rangle=\bigoplus_{i=1}^t \MB_{F_i}$ and $t$ is as large as possible. Note that $t$ exists since this equation holds for the trivial partition $\{E\}$. The condition that $t$ is as large as possible ensures that the family of subspaces $\MB_{F_i}$ is connected for all $1\leq i\leq t$.
We can now apply \Cref{thm:connected} to each $\MA_{F_i}$ to deduce that equality holds in the second inequality of 
 (\ref{eq:dil1}). 
\end{proof}

\end{document}